\documentclass[12pt,a4paper]{amsart}

\usepackage{amssymb}
\usepackage[numeric, abbrev, nobysame]{amsrefs}
\usepackage{amscd}

\def\frak{\mathfrak}
\def\Bbb{\mathbb}
\def\Cal{\mathcal}

\let\phi\varphi

\newcommand{\x}{\times}
\renewcommand{\o}{\circ}

\newcommand{\al}{\alpha}

\newcommand{\la}{\lambda}

\newcommand{\ph}{\phi}

\renewcommand{\th}{\theta}
\newcommand{\si}{\sigma}

\newcommand{\La}{\Lambda}
\newcommand{\Ph}{\Phi}
\newcommand{\Ps}{\Psi}

\newcommand{\Si}{\Sigma}

\newcommand{\im}{\operatorname{im}}

\newcommand{\gr}{\operatorname{gr}}
\newcommand{\id}{\operatorname{id}}

\newcommand{\tcf}{{\widetilde{\Cal F}}}

\newcounter{theorem}
\numberwithin{theorem}{section}
\numberwithin{equation}{section}

\newtheorem{thm}[theorem]{Theorem}
\newtheorem*{thm*}{Theorem \thesubsection}
\newtheorem{lemma}[theorem]{Lemma}
\newtheorem{prop}[theorem]{Proposition}
\newtheorem{cor}[theorem]{Corollary}
\newtheorem*{lemma*}{Lemma \thesubsection}
\newtheorem*{prop*}{Proposition \thesubsection}
\newtheorem*{cor*}{Corollary \thesubsection}

\theoremstyle{definition}
\newtheorem{definition}[theorem]{Definition}
\newtheorem*{definition*}{Definition \thesubsection}
\newtheorem{example}[theorem]{Example}
\newtheorem*{example*}{Example \thesubsection}
\theoremstyle{remark}

\newtheorem*{remark*}{Remark \thesubsection}

\newcommand\dxd{%
\begin{picture}(46,10)\put(5,3){\line(1,0){36}}%
\put(3,3){\makebox(0,0){$\o$}}\put(23,3){\makebox(0,0){$\times$}}%
\put(43,3){\makebox(0,0){$\o$}}\end{picture}}

\newcommand\xxd[3]{%
\begin{picture}(46,12)\put(3,3){\line(1,0){38}}%
\put(3,3){\makebox(0,0){$\times$}}\put(23,3){\makebox(0,0){$\times$}}%
\put(43,3){\makebox(0,0){$\o$}}%
\put(3,10){\makebox(0,0){\scriptsize $#1$}}%
\put(23,10){\makebox(0,0){\scriptsize $#2$}}%
\put(43,10){\makebox(0,0){\scriptsize $#3$}}\end{picture}}

\newcommand\wxxd[3]{%
\begin{picture}(86,12)\put(3,3){\line(1,0){78}}%
\put(3,3){\makebox(0,0){$\times$}}\put(43,3){\makebox(0,0){$\times$}}%
\put(83,3){\makebox(0,0){$\o$}}%
\put(3,10){\makebox(0,0){\scriptsize $#1$}}%
\put(43,10){\makebox(0,0){\scriptsize $#2$}}%
\put(83,10){\makebox(0,0){\scriptsize $#3$}}\end{picture}}

\newcommand\waxxd[3]{%
\begin{picture}(76,12)\put(3,3){\line(1,0){68}}%
\put(3,3){\makebox(0,0){$\times$}}\put(33,3){\makebox(0,0){$\times$}}%
\put(73,3){\makebox(0,0){$\o$}}%
\put(3,10){\makebox(0,0){\scriptsize $#1$}}%
\put(33,10){\makebox(0,0){\scriptsize $#2$}}%
\put(73,10){\makebox(0,0){\scriptsize $#3$}}\end{picture}}

\newcommand\wbxxd[3]{%
\begin{picture}(66,12)\put(3,3){\line(1,0){58}}%
\put(3,3){\makebox(0,0){$\times$}}\put(43,3){\makebox(0,0){$\times$}}%
\put(63,3){\makebox(0,0){$\o$}}%
\put(3,10){\makebox(0,0){\scriptsize $#1$}}%
\put(43,10){\makebox(0,0){\scriptsize $#2$}}%
\put(63,10){\makebox(0,0){\scriptsize $#3$}}\end{picture}}

\newcommand\wcxxd[3]{%
\begin{picture}(56,12)\put(3,3){\line(1,0){48}}%
\put(3,3){\makebox(0,0){$\times$}}\put(33,3){\makebox(0,0){$\times$}}%
\put(53,3){\makebox(0,0){$\o$}}%
\put(3,10){\makebox(0,0){\scriptsize $#1$}}%
\put(33,10){\makebox(0,0){\scriptsize $#2$}}%
\put(53,10){\makebox(0,0){\scriptsize $#3$}}\end{picture}}

\newcommand\Wxxd[3]{%
\begin{picture}(106,12)\put(3,3){\line(1,0){98}}%
\put(3,3){\makebox(0,0){$\times$}}\put(63,3){\makebox(0,0){$\times$}}%
\put(103,3){\makebox(0,0){$\o$}}%
\put(3,10){\makebox(0,0){\scriptsize $#1$}}%
\put(63,10){\makebox(0,0){\scriptsize $#2$}}%
\put(103,10){\makebox(0,0){\scriptsize $#3$}}\end{picture}}

\newcommand\xdd[3]{%
\begin{picture}(46,12)\put(3,3){\line(1,0){18}}\put(25,3){\line(1,0){16}}%
\put(3,3){\makebox(0,0){$\times$}}\put(23,3){\makebox(0,0){$\o$}}%
\put(43,3){\makebox(0,0){$\o$}}%
\put(3,10){\makebox(0,0){\scriptsize $#1$}}%
\put(23,10){\makebox(0,0){\scriptsize $#2$}}%
\put(43,10){\makebox(0,0){\scriptsize $#3$}}\end{picture}}

\newcommand\wxdd[3]{%
\begin{picture}(66,12)\put(3,3){\line(1,0){38}}\put(45,3){\line(1,0){16}}%
\put(3,3){\makebox(0,0){$\times$}}\put(43,3){\makebox(0,0){$\o$}}%
\put(63,3){\makebox(0,0){$\o$}}%
\put(3,10){\makebox(0,0){\scriptsize $#1$}}%
\put(43,10){\makebox(0,0){\scriptsize $#2$}}%
\put(63,10){\makebox(0,0){\scriptsize $#3$}}\end{picture}}

\newcommand\awxdd[3]{%
\begin{picture}(76,12)\put(3,3){\line(1,0){38}}\put(45,3){\line(1,0){26}}%
\put(3,3){\makebox(0,0){$\times$}}\put(43,3){\makebox(0,0){$\o$}}%
\put(73,3){\makebox(0,0){$\o$}}%
\put(3,10){\makebox(0,0){\scriptsize $#1$}}%
\put(43,10){\makebox(0,0){\scriptsize $#2$}}%
\put(73,10){\makebox(0,0){\scriptsize $#3$}}\end{picture}}

\newcommand\bwxdd[3]{%
\begin{picture}(66,12)\put(3,3){\line(1,0){28}}\put(35,3){\line(1,0){26}}%
\put(3,3){\makebox(0,0){$\times$}}\put(33,3){\makebox(0,0){$\o$}}%
\put(63,3){\makebox(0,0){$\o$}}%
\put(3,10){\makebox(0,0){\scriptsize $#1$}}%
\put(33,10){\makebox(0,0){\scriptsize $#2$}}%
\put(63,10){\makebox(0,0){\scriptsize $#3$}}\end{picture}}

\newcommand\ddd[3]{%
\begin{picture}(46,12)\put(5,3){\line(1,0){16}}\put(25,3){\line(1,0){16}}%
\put(3,3){\makebox(0,0){$\o$}}\put(23,3){\makebox(0,0){$\o$}}%
\put(43,3){\makebox(0,0){$\o$}}%
\put(3,10){\makebox(0,0){\scriptsize $#1$}}%
\put(23,10){\makebox(0,0){\scriptsize $#2$}}%
\put(43,10){\makebox(0,0){\scriptsize $#3$}}\end{picture}}

\def\sideremark#1{\ifvmode\leavevmode\fi\vadjust{\vbox to0pt{\vss
 \hbox to 0pt{\hskip\hsize\hskip1em
 \vbox{\hsize3cm\tiny\raggedright\pretolerance10000
  \noindent #1\hfill}\hss}\vbox to8pt{\vfil}\vss}}}%
                        
\begin{document}
\renewcommand{\today}{}
\title{Relative BGG sequences;\\I.\ Algebra}

\author{Andreas \v Cap and Vladimir Sou\v cek}

\address{A.\v C.: Faculty of Mathematics\\
University of Vienna\\
Oskar--Morgenstern--Platz 1\\
1090 Wien\\
Austria\\
V.S.:Mathematical Institute\\ Charles University\\ Sokolovsk\'a
  83\\Praha\\Czech Republic} 
\email{Andreas.Cap@univie.ac.at}
\email{soucek@karlin.mff.cuni.cz}


\begin{abstract}
We develop a relative version of Kostant's harmonic theory and use
this to prove a relative version of Kostant's theorem on Lie algebra
(co)homology. These are associated to two nested parabolic subalgebras
in a semisimple Lie algebra. We show how relative homology groups can
be used to realize representations with lowest weight in one (regular
or singular) affine Weyl orbit. In the regular case, we show how all
the weights in the orbit can be realized as relative homology groups
(with different coefficients). These results are motivated by
applications to differential geometry and the construction of
invariant differential operators. 
\end{abstract}

\thanks{First author supported by projects P23244--N13 and P27072--N25
  of the Austrian Science Fund (FWF), second author supported by the
  grant P201/12/G028 of the Grant Agency of the Czech Republic (GACR)}

\maketitle

\pagestyle{myheadings}\markboth{\v Cap and Sou\v cek}{Relative
  BGG--Sequences I}

\section{Introduction}\label{1}
This article is the first in a series of two. The main aim of the
series is to develop a relative version of the machinery of
Bernstein--Gelfand--Gelfand sequences (or BGG sequences) as introduced
in \cite{CSS-BGG} and \cite{Calderbank--Diemer} and to improve the
original constructions at the same time, which is done in
\cite{part2}. This is a construction for invariant differential
operators associated to a class of geometric structures known as
parabolic geometries. For each type of parabolic geometries, there is
a homogeneous model, which is a generalized flag manifold, i.e.~the
quotient of a (real or complex) semisimple Lie group $G$ by a
parabolic subgroup $P$. The starting point for the construction of a
BGG sequence is a finite--dimensional representation $\Bbb V$ of
$G$. On the homogeneous model, the resulting sequence is a resolution
of the locally constant sheaf $\Bbb V$ on $G/P$ by differential
operators acting on spaces of sections of homogeneous vector bundles
induced by irreducible representations of $P$. The resulting
resolution of $\Bbb V$ by principal series representations of $G$ is
dual (in a certain sense) to the resolution of $\Bbb V^*$ by
generalized Verma--modules obtained in \cite{Lepowsky}. This
generalizes the Bernstein--Gelfand---Gelfand resolution of $\Bbb V^*$
by Verma--modules from \cite{BGG}, which motivated the name of the
construction.

The algebraic character of BGG sequences is also reflected by the
tools needed for their construction, and this first part of the series
is devoted to developing the necessary algebraic background 
for the relative version. In particular, we prove a relative version
of Kostant's theorem on Lie algebra cohomology from \cite{Kostant},
which should be of independent interest. The setup for Kostant's
original theorem is a complex semisimple Lie algebra $\frak g$, a
parabolic subalgebra $\frak p\subset\frak g$ with (reductive)
Levi--decomposition $\frak p=\frak g_0\oplus\frak p_+$ and a complex
irreducible representation $\Bbb V$ of $\frak g$. Then Kostant
considered the standard complex $(C^*(\frak p_+,\Bbb V),\partial)$
computing the Lie algebra cohomology of the nilpotent Lie algebra
$\frak p_+$ with coefficients in $\Bbb V$. The spaces in this complex
are naturally representations of $\frak g_0$ and the differentials are
$\frak g_0$--equivariant. Thus the cohomology groups $H^*(\frak
p_+,\Bbb V)$ are representations of the reductive Lie algebra $\frak
g_0$ and Kostant's theorem describes these representations explicitly
and algorithmically in terms of highest weights.

While higher Lie algebra cohomology groups seem to be difficult to
interpret in general, Kostant's theorem has immediate algebraic
applications, see \cite{Kostant}. Even the version for the Borel
subalgebra (which in some respects is significantly simpler than the
general result) very quickly implies the Weyl character formula, thus
providing a completely algebraic proof for this formula. Moreover,
together with the Peter--Weyl theorem, Kostant's theorem can be used
to proof Bott's generalized version (see \cite{Bott}) of the
Borel--Weil theorem describing the sheaf cohomology of the sheaf of
local holomorphic sections of a homogeneous vector bundle over a
complex generalized flag manifold. Apart from the applications in the
theory of parabolic geometries (see \cite{book}), Kostant's theorem
has also been applied in other areas recently. For example, in
\cite{LR}, Lie algebra cohomology as computed via Kostant's theorem is
used as a replacement for Spencer cohomology in connection with
exterior differential systems to prove rigidity results in algebraic
geometry.

It is important to point out here that for the applications to
parabolic geometries, not only Kostant's theorem itself is
important. Also some of the tools introduced by Kostant in the proof
play a central role there. These tools are also available for
parabolic subalgebras in real semisimple Lie algebras, and the real
versions are needed in the applications. Moreover, for these
applications it is very important to carefully keep track about the
possibility of formulating things in a $\frak p$--invariant
way. Likewise, one has to carefully distinguish between filtered
modules and their associated graded modules in this setting, while
this distinction plays no role for $\frak g_0$--modules. In view of
the applications in \cite{part2}, we will work in a real setting for
most of the article and be more careful about invariance under the
parabolic subalgebra than it would be necessary for the purposes of
the current article.

The setup for the relative version is provided by two nested parabolic
subalgebras $\frak q\subset\frak p$ in a semisimple Lie algebra $\frak
g$. Here $\frak q$ will be the ``main'' parabolic subalgebra (so
$\frak q$--invariance will be important), while the larger parabolic
$\frak p$ is an auxiliary input. If $\Bbb V$ is an irreducible
representation of $\frak p$, then the nilradical $\frak
p_+\subset\frak p$ acts trivially on $\Bbb V$. Moreover, the
nilradicals satisfy $\frak p_+\subset\frak q_+\subset\frak p$, so we
can naturally view $\Bbb V$ as a representation of the nilpotent Lie
algebra $\frak q_+/\frak p_+$. In view of $\frak q$--invariance, it is
preferable to work with Lie algebra homology rather than Lie algebra
cohomology. The standard complex for Lie algebra homology consists of
$\frak q$--modules and $\frak q$--equivariant maps, so the Lie algebra
homology groups $H_*(\frak q_+/\frak p_+,\Bbb V)$ are automatically
representations of $\frak q$. These are completely reducible, and
hence can (in the complex case) be described via weights. This
description is the content of the relative version of Kostant's
theorem, which we prove as Theorem \ref{thm2.9}. As a module for the
Levi--factor of $\frak q$, one can identify this with Lie algebra
cohomology for the quotients of the nilradicals of the opposite
parabolics. As in Kostant's original theorem, the description is in
terms of the orbit of an extremal weight of $\Bbb V$ under a subset
$W^{\frak q}_{\frak p}$ of the Weyl group of $\frak g$.

The second main topic of the article is the relation between absolute
and relative homology groups in the case of regular infinitesimal
character. Suppose that, given $\frak q\subset\frak p\subset\frak g$,
we take an irreducible representation $\tilde{\Bbb V}$ of $\frak g$
and let $\Bbb V$ be its $\frak p$--irreducible quotient, which can be
realized as $H_0(\frak p_+,\tilde{\Bbb V})$. Then it follows from the
explicit descriptions via the Weyl group that $H_*(\frak q_+/\frak
p_+,\Bbb V)$ consists of some of the irreducible components of
$H_*(\frak q_+,\tilde{\Bbb V})$. More generally, we prove that the
Hasse diagram $W^{\frak q}$ of $\frak q$, which parametrizes the
irreducible components in $H_*(\frak q_+,\tilde{\Bbb V})$ can be
written as a product $W^{\frak q}_{\frak p}\x W^{\frak p}$, which
leads to an isomorphism
$$
H_*(\frak q_+,\tilde{\Bbb V})\cong H_*(\frak q_+/\frak p_+,H_*(\frak
p_+,\tilde{\Bbb V})),
$$ see Theorem \ref{thm3.3}. The description as a product
significantly simplifies the determination of the Hasse diagram for
non--maximal parabolic subalgebras as we demonstrate in Example
\ref{ex3.2}. Initially, the above isomorphism of homology groups is
proved by comparing lowest weights. In preparation for the
applications in \cite{part2}, we conclude the paper by constructing an
explicit isomorphism from $\frak q$--invariant data. This is based on
a filtration of $C_*(\frak q_+,\tilde{\Bbb V})$ by $\frak
q$--submodules. While this does not lead to a filtered complex, it can
be used to construct explicit $\frak q$--equivariant maps, which
realize the decomposition of $H_*(\frak q_+,\tilde{\Bbb V})$ according
to the degree in the second factor under the above isomorphism.

\section{A relative version of Kostant's theorem}\label{2}
Given two nested parabolic subalgebras $\frak q\subset\frak p$ in a
(real or complex) semisimple Lie algebra $\frak g$, we first develop a
relative version of Kostant's Hodge theory. As remarked in the
introduction, we work in a $\frak q$--invariant setting as much as
possible. Given a completely reducible representation $\Bbb V$ of
$\frak p$, the main object of study thus is the standard complex for
Lie algebra homology, while the other ingredients for the Hodge theory
are of auxiliary nature. Then we compute the action of the algebraic
Laplacian on an isotypical component. In the complex case, this leads
to a description of Lie algebra homology in terms of a relative analog
of the Hasse diagram.

\subsection{Nested parabolic subalgebras}\label{2.1}
Throughout this article, we consider a (real or complex) semisimple Lie
algebra $\frak g$ endowed with two nested parabolic subalgebras $\frak
q\subset\frak p\subset\frak g$. (In view of the applications we have
in mind, this notation, as well as much of the notation in what
follows is chosen in accordance with the literature on parabolic
geometries.) It is well known that both in the real and in the complex
case, parabolic subalgebras can be described in terms of so--called
$|k|$--gradings of $\frak g$ (for various values of $k\in\mathbb
N$). Such a grading is a decomposition of $\frak g$ as a direct sum
$$
\frak g=\frak g_{-k}\oplus\dots\oplus\frak g_k
$$ 
such that $[\frak g_i,\frak g_j]\subset\frak g_{i+j}$ (with $\frak
g_\ell=\{0\}$ for $|\ell|>k$). Moreover, we make the technical
assumptions that no simple ideal of $\frak g$ is contained in $\frak
g_0$ (which is a subalgebra of $\frak g$ by the grading property) and
that the positive (negative) part of the grading is generated as a Lie
algebra by $\frak g_1$ ($\frak g_{-1}$). The parabolic subalgebra
determined by such a grading is the non--negative part $\frak
g_0\oplus\dots\oplus\frak g_k$ of the grading. It then turns out that
the positive part is the nilradical of the parabolic subalgebra, while
$\frak g_0$ is its Levi factor. 

Since we have to deal with two parabolics at the same time, we will
avoid the use of explicit indices for the grading
components. Following \cite{twistor}, we denote the decompositions of
$\frak g$ determined by the two parabolic subalgebras by
$$
\frak g=\frak q_-\oplus\frak q_0\oplus\frak q_+ \quad\text{and}\quad 
\frak g=\frak p_-\oplus\frak p_0\oplus\frak p_+, 
$$ respectively. For a parabolic subalgebra, the nilradical coincides
with the annihilator with respect to the Killing form. Thus $\frak
q\subset\frak p$ implies $\frak p_+\subset\frak q_+$ and by symmetry
of the grading we conclude that $\frak p_-\subset\frak q_-$. Since
$\frak p_+$ is an ideal in $\frak p$, it is also an ideal in $\frak
q$ and in $\frak q_+$.

The classification of parabolic subalgebras is done via structure
theory. Given a $|k|$--grading of a complex semisimple Lie algebra
$\frak g$, one shows that one can choose a Cartan subalgebra $\frak h$
contained in $\frak g_0$, so its adjoint action preserves the
decomposition defined by the grading. The assumptions then easily
imply that the root spaces corresponding to simple roots are either
contained in $\frak g_0$ or in $\frak g_1$. Denoting by
$\Si\subset\Delta^0$ the set of those simple roots with root spaces
contained in $\frak g_1$, it turns out that the grading is given by
the $\Si$--height. This means that to determine the grading component
containing a root space $\frak g_{\al}$, one expresses $\al$ as a
linear combination of simple roots and then adds up the coefficients
of the roots contained in $\Si$. This leads to a bijective
correspondence between the set of conjugacy classes of parabolic
subalgebras and subsets of $\Delta^0$, with the empty set
corresponding to $\frak g$ and the full set $\Delta^0$ corresponding
to the Borel subalgebras. In the real case, one obtains a similar
description in terms of restricted roots for a maximally non--compact
Cartan subalgebra, see section 3.2.9 of \cite{book}.

The grading components for a $|k|$--grading can be realized as the
eigenspaces for the adjoint action of an element in the center of
$\frak g_0$, the so--called grading element. In terms of structure
theory, this element is always contained in the Cartan subalgebra. In
particular, the grading elements associated to two nested parabolics
commute, so the gradings are compatible. In particular, this implies
that $\frak p_0$ is invariant under the adjoint action of the grading
element $E_{\frak q}$ for $\frak q$. Hence we obtain the finer
decomposition
\begin{equation}\label{g-decomp}
\frak g=\frak p_-\oplus(\frak p_0\cap\frak q_-)\oplus\frak
q_0\oplus(\frak p_0\cap\frak q_+)\oplus\frak p_+ 
\end{equation}
of $\frak g$ into a direct sum of subalgebras with the first two
summands adding up to $\frak q_-$ and the last two summands adding up
to $\frak q_+$. 

Let us finally remark that it is no problem to choose parabolic
subgroups corresponding to two nested parabolic subalgebras in a nice
way. Given a Lie group $G$ with Lie algebra $\frak g$, we can first
choose a parabolic subgroup $P\subset G$ corresponding to $\frak
p\subset\frak g$. This means that $P$ is a subgroup of the normalizer
$N_G(\frak p)$ which contains the connected component of the identity
of this normalizer. Then we consider the normalizer $N_P(\frak q)$ of
$\frak q$ in $P$ and choose a subgroup $Q$ lying between this
normalizer and its connected component of the identity, thus obtaining
groups $Q\subset P\subset G$ corresponding to $\frak q\subset\frak
p\subset\frak g$. 

\subsection{Lie algebra homology}\label{2.2}
While the general theory of finite dimensional representations of a
parabolic subalgebra is rather intractable, irreducible
representations (and hence also completely reducible ones) are rather
easy to understand. If $\frak p\subset \frak g$ is a parabolic
subalgebra and $\Bbb V$ is an irreducible representation of $\frak p$,
then the nilradical $\frak p_+$ acts trivially on $\Bbb V$, so the
representation descends to the reductive quotient $\frak p/\frak
p_+\cong\frak p_0$. Conversely, a representation of $\frak p_0$ is
completely reducible, provided that the center acts diagonalizably,
and one can extend such a representation to $\frak p$ (with $\frak
p_+$ acting trivially).

In the setting of two nested parabolic subalgebras $\frak
q\subset\frak p$ in $\frak g$ and a completely reducible
representation $\Bbb V$ of $\frak p$, we can first restrict the
representation to the (nilpotent) subalgebra $\frak q_+\subset\frak
p$. As we have noted in Section \ref{2.1}, $\frak q_+$ contains the
nilradical $\frak p_+$ of $\frak p$, and clearly $\frak p_+$ is an
ideal in $\frak q_+$. Hence $\Bbb V$ is automatically a representation
of the Lie algebra $\frak q_+/\frak p_+$. Thus we can consider the
standard complex for computing Lie algebra homology of $\frak
q_+/\frak p_+$ with coefficients in $\Bbb V$. The spaces in this
complex are
$$
C_k(\frak q_+/\frak p_+,\Bbb V):=\La^k(\frak q_+/\frak p_+)\otimes\Bbb
V. 
$$ 
Following \cite{Kostant}, where the differential in this complex was
obtained as the adjoint of a Lie algebra cohomology differential, we
denote it by
$$ 
\partial^*_\rho : \La^k(\frak q_+/\frak p_+)\otimes\Bbb V\to
\La^{k-1}(\frak q_+/\frak p_+)\otimes\Bbb V, 
$$ and call it the \textit{relative Kostant--codifferential}.
Explicitly, this differential is given by
\begin{equation}\label{eq:codiff}
\begin{aligned}
 \partial^*_\rho(Z_1&\wedge\dots\wedge Z_k\otimes v):=
\textstyle\sum_i(-1)^iZ_1\wedge\dots\wedge\widehat{Z_i}\wedge\dots\wedge
Z_k\otimes Z_i\cdot v\\ 
&+\textstyle\sum_{i<j}(-1)^{i+j}[Z_i,Z_j]\wedge
Z_1\wedge\dots\wedge\widehat{Z_i}\wedge\dots\wedge\widehat{Z_j}\wedge
\dots\wedge Z_k\otimes v, 
\end{aligned}
\end{equation}
for $Z_1,\dots,Z_k\in \frak q_+/\frak p_+$ and $v\in \Bbb V$, with
hats denoting omission. It satisfies
$\partial^*_\rho\o\partial^*_\rho=0$ and the homology groups of
$(C_*(\frak q_+/\frak p_+,\Bbb V),\partial^*_\rho)$ are the Lie
algebra homology groups $H_*(\frak q_+/\frak p_+,\Bbb V)$.

\begin{prop}\label{prop2.2}
  Let $\frak q\subset\frak p\subset\frak g$ be nested parabolic
  subalgebras, $Q\subset P\subset G$ be corresponding groups as in
  Section \ref{2.1}, and let $\Bbb V$ be a finite dimensional,
  completely reducible representation of $\frak p$.

  Then the spaces $C_*(\frak q_+/\frak p_+,\Bbb V)$ are naturally
  $\frak q$--modules such that the differentials $\partial^*_\rho$ are
  $\frak q$--equivariant. Moreover, this action has the property that
  $\frak q_+\cdot\ker(\partial^*_\rho)\subset\im(\partial^*_\rho)$, so
  the induced representations of $\frak q$ on the homology groups
  $H_*(\frak q_+/\frak p_+,\Bbb V)$ are completely reducible. Finally,
  if $\Bbb V$ is a completely reducible representation of the group
  $P$, then the module structures lift to $Q$ and the differentials
  are $Q$--equivariant.
\end{prop}
\begin{proof}
  As we have noted in Section \ref{2.1}, the adjoint action of $Q$
  preserves both $\frak q_+$ and $\frak p_+$, so there is a natural
  induced action (by automorphisms) on $\frak q_+/\frak
  p_+$. Correspondingly, $\frak q$ acts on $\frak q_+/\frak p_+$ by
  derivations. On the other hand, $\frak q$ acts on $\Bbb V$ by
  restriction of the $\frak p$--action, and an action of $P$ on $\Bbb
  V$ can be restricted to the subgroup $Q$. Hence we obtain the
  claimed module structures on the spaces $C_*(\frak q_+/\frak
  p_+,\Bbb V)$. Of course the action $(\frak q_+/\frak p_+)\x\Bbb
  V\to\Bbb V$ is equivariant for the natural $\frak q$--action and for
  the $Q$--action in case that $\Bbb V$ is a $P$--module. Together
  with the above this implies equivariancy of the differentials
  $\partial^*_\rho$ by their definition.

  Starting from the definition of $\partial^*_\rho$, a simple
  computation (see the proof of Lemma 3.3.2 in \cite{book}) shows that
  for $\ph\in\La^k(\frak q_+/\frak p_+)\otimes\Bbb V$ and $Z\in\frak
  q_+$ one gets
$$ 
Z\cdot \ph=-\partial^*_\rho((Z+\frak p_+)\wedge\ph)-(Z+\frak
p_+)\wedge\partial^*_\rho(\ph).
$$ 
This immediately implies that $\frak
q_+\cdot\ker(\partial^*_\rho)\subset\im(\partial^*_\rho)$, so $\frak
q_+$ acts trivially on the homology groups. Since by assumption the
center of $\frak q_0$ acts diagonalizably on all the involved
representations, complete reducibility of the homology representations
follows.
\end{proof}

\subsection{Lie algebra cohomology}\label{2.3}
The first step towards the computation of the Lie algebra homology
groups $H_*(\frak q_+/\frak p_+,\Bbb V)$ is the construction of an
adjoint to the Lie algebra homology differential. However, such an
adjoint cannot be constructed as a $\frak q$--equivariant map, one
only obtains equivariancy under $\frak q_0$.  

The Killing form $B$ on $\frak g$ is compatible with the
$|k|$--grading determined by any parabolic subalgebra in the sense that
it induces a duality (of $\frak g_0$--modules) between $\frak g_i$ and
$\frak g_{-i}$ for all $i\neq 0$ and its restriction to $\frak g_0$ is
non--degenerate. For the two decompositions we are dealing with, this
can be interpreted as an identification of $\frak q_+$ with $\frak
q_{-}^*$ while $\frak p_+\subset\frak q_+$ can be identified with the
annihilator of $\frak p\cap\frak q_-=\frak p_0\cap\frak q_-$. Hence we
can identify $\frak q_+/\frak p_+$, as a $\frak q_0$--module, with the
dual of $\frak p_0\cap\frak q_-\subset\frak g$. Since both $\frak p_0$
and $\frak q_-$ are Lie subalgebras of $\frak g$ and $\frak q_-$ is
nilpotent, $\frak p_0\cap\frak q_-$ is a nilpotent Lie subalgebra of
$\frak g$, which naturally acts on $\Bbb V$ by the restriction of the
$\frak p_0$--action.

Hence, for each $k$, we can identify the chain group $C_k(\frak
q_+/\frak p_+,\Bbb V)$ as a $\frak q_0$--module with the cochain group
$C^k(\frak q_-\cap\frak p_0,\Bbb V)$. Consequently, we obtain the Lie
algebra cohomology differential, which for consistency we denote by
$$
\partial_\rho:C_k(\frak p_+/\frak q_+,\Bbb V)\to C_{k+1}(\frak p_+/\frak
q_+,\Bbb V).  
$$
Viewing elements of $C_*(\frak p/\frak q,\Bbb V)$ as alternating
multilinear maps from $\frak q_-\cap\frak p_0$ to $\Bbb V$, this
differential is given by
\begin{align*}
  \partial_\rho\ph(X_0&,\dots,X_k):=\textstyle\sum_{i=0}^k(-1)^iX^i\cdot
  \ph(X_0,\dots,\widehat{X_i},\dots,X_k)\\
  &+\textstyle\sum_{i<j}(-1)^{i+j}\ph([X_i,X_j],X_0,\dots,\widehat{X_i},\dots,
  \widehat{X_j},\dots,X_k),
\end{align*}
for $X_0,\dots,X_k\in\frak q_-\cap\frak p_0$ with hats denoting
omission.

\subsection{Algebraic Hodge decomposition}\label{2.4}
The first key step toward the proof of Kostant's theorem and its
relative analog is that the two differentials $\partial^*_\rho$ and
$\partial_\rho$ satisfy a property called disjointness by
Kostant. This easily implies that they give rise to a
Hodge--decomposition of the chain groups. To formulate this, we first
define the obvious analog
$\square_\rho=\partial^*_\rho\o\partial_\rho+\partial_\rho\o\partial^*_\rho$
of the Kostant Laplacian, which maps each $C_k(\frak q_+/\frak p_+,\Bbb V)$
to itself.

\begin{lemma}[Hodge decomposition]\label{lem2.4}
  For any completely reducible representation $\Bbb V$ of $\frak p$
  and each $k=0,\dots,\dim(\frak q_+/\frak p_+)$ one has a
  decomposition
$$
C_k(\frak q_+/\frak p_+,\Bbb
V)=\im(\partial^*_\rho)\oplus\ker(\square_\rho)\oplus\im(\partial_\rho) 
$$
as a direct sum of $\frak q_0$--modules. Moreover, the first two
summands add up to $\ker(\partial^*_\rho)$, while the last two
summands add up to $\ker(\partial_\rho)$. Consequently, both the Lie
algebra homology group $H_k(\frak q_+/\frak p_+,\Bbb V)$ and the Lie
algebra cohomology group $H^k(\frak q_-\cap\frak p_0,\Bbb V)$ are
isomorphic to $\ker(\square_\rho)\subset C_k(\frak q_+/\frak p_+,\Bbb
V)$ as $\frak q_0$--modules.
\end{lemma}
\begin{proof}
  The main step in the proof is to verify disjointness of the
  operators $\partial_\rho$ and $\partial^*_\rho$ in the sense of
  Kostant, i.e.~that
  $\ker(\partial^*_\rho)\cap\im(\partial_\rho)=\{0\}$ and
  $\ker(\partial_\rho)\cap\im(\partial^*_\rho)=\{0\}$. Having verified
  this, the argument in the proof of Theorem 3.3.1 of \cite{book} can
  be applied without changes to prove the Hodge decomposition.

  Decomposing $\Bbb V$ into a direct sum of irreducibles, we get an
  induced decomposition of all chain spaces, which is preserved by
  both operators, so it suffices to prove disjointness in the case
  that $\Bbb V$ is irreducible. Moreover, disjointness of two maps
  clearly follows from disjointness of complex linear
  extensions. Using complexifications, we may thus without loss of
  generality assume that $\frak g$ is a complex semisimple Lie
  algebra, $\frak q\subset\frak p\subset \frak g$ are complex
  parabolic subalgebras, and that $\Bbb V$ is a complex irreducible
  representation of $\frak p$. This means that $\Bbb V$ is an
  irreducible representation of the reductive Lie algebra $\frak p_0$
  (extended trivially on $\frak p_+$), so in particular the center
  $\frak z(\frak p_0)$ acts by scalars determined by a complex linear
  functional $\la:\frak z(\frak p_0)\to\Bbb C$.

  We may further assume that we have chosen a Cartan subalgebra $\frak
  h\subset\frak g$ and a set of positive roots such that $\frak p$ and
  $\frak q$ are standard parabolics with respect to these choices, so
  they both contain $\frak h$ and all positive root spaces. Now we
  verify disjointness by constructing a positive definite inner
  product on the chain spaces for which the operators
  $\partial^*_\rho$ and $\partial_\rho$ are adjoint. To do this, we
  first consider the Cartan involution $\th$ for $\frak g$ coming from
  the standard construction of a compact real form $\frak
  u\subset\frak g$ as in Proposition 2.3.1 of \cite{book}. This acts
  as minus the identity on the real subspace $\frak h_0\subset\frak h$
  on which all roots have real values, so that for the grading
  elements we get $\th(E_{\frak p})=-E_{\frak p}$ and $\th(E_{\frak
    q})=-E_{\frak q}$. In particular, $\th$ respects both $\frak p_0$
  and $\frak q_0$ and exchanges $\frak p_0\cap\frak q_-$ and $\frak
  p_0\cap\frak q_+$. Moreover, the Killing form $B$ of $\frak g$ is
  non--degenerate on $(\frak p_0\cap\frak q_-)\oplus(\frak
  p_0\cap\frak q_+)$ which shows that $B_\th(X,Y):=-B(X,\th(Y))$
  restricts to a positive definite inner product on $\frak
  p_0\cap\frak q_+$. Hence we also get an induced inner product on
  $\La^*(\frak p_0\cap\frak q_+)$.

  As an automorphism of the Lie algebra $\frak p_0$, $\th$ also
  respects the decomposition $\frak p_0=\frak z(\frak p_0)\oplus
  [\frak p_0,\frak p_0]$, and we denote the second summand by $\frak
  p_0^s$. The fixed point set of $\th|_{\frak p_0^s}$ is $\frak
  u\cap\frak p_0^s$ so this is a compact real form of $\frak
  p_0^s$. In case that the functional $\la$ is non--zero, we next have
  to modify the restriction of $\th$ to $\frak z(\frak p_0)$ in such a
  way that $\la\o \th=-\bar\la$, for example by choosing an
  isomorphism with $\Bbb C^k$ with first coordinate $\la$ and then
  pulling back complex conjugation.

  Compactness of $\frak u\cap\frak p_0^s$ then implies that there is a
  Hermitian inner product $\langle\ ,\ \rangle$ on $\Bbb V$ for which
  elements of this subalgebra act by skew Hermitian endomorphisms. Now
  by construction, $\langle A\cdot v_1,v_2\rangle=-\langle
  v_1,\th(A)\cdot v_2\rangle$. Together with the inner product on
  $\La^*(\frak p_0\cap\frak p_+)$ from above, we get inner products on
  all chain spaces. Having these inner products at hand, the proof of
  adjointness of $\partial_\rho$ and $\partial^*_\rho$ works exactly
  as in Proposition 3.1.1 of \cite{book}, and disjointness follows.
\end{proof}

\subsection{A $\frak q$--invariant interpretation}\label{2.5}
At this point, we make a short digression, which is mainly needed for
the geometric applications in \cite{part2}. For these applications,
we need interpretations of what we have done so far in terms of $\frak
q$--modules. As we have noted in Section \ref{2.2}, this is not a
problem for the complex $(C_*(\frak q_+/\frak p_+,\Bbb
V),\partial^*_\rho)$, which consists of $\frak q$--modules and $\frak
q$--equivariant maps.

However, the Lie subalgebra $\frak p_0\cap\frak q_-$ used in the
construction of $\partial_\rho$ does not carry a natural $\frak
q$--module structure. Of course, one could define such a structure via
the $\frak q_0$--equivariant isomorphism with $(\frak q_+/\frak
p_+)^*$ from Section \ref{2.3}, but then the Lie bracket is not $\frak
q$--equivariant. This is reflected in the fact that, viewed as maps on
the spaces $C_*(\frak q_+/\frak p_+,\Bbb V)$, the Lie algebra
cohomology differentials $\partial_\rho$ are only $\frak
q_0$--equivariant and not $\frak q$--equivariant. 

To solve this problem, we first observe that on a completely reducible
representation $\Bbb V$ of $\frak p$, one obtains a natural grading
similar to the $|k|$--grading on $\frak g$. In the case that $\Bbb V$
is a complex irreducible representation of $\frak p$, we can view it
as a representation of the reductive algebra $\frak p_0$. This means
that elements of the Cartan subalgebra act diagonalizably on $\Bbb V$,
so in particular this is true for the grading element $E_{\frak q}$
which lies in the center of $\frak q_0$. Hence we can decompose $\Bbb
V$ into eigenspaces for $E_{\frak q}$, which all are $\frak
q_0$--invariant by construction. With respect to the grading of $\frak
g=\oplus\frak g_i$ induced by $\frak q$, this clearly has the property
that the action of $\frak g_j$ maps the eigenspace with eigenvalue $a$
to the eigenspace for $a+j$ for any $j$. So we can view the eigenspace
decomposition of $\Bbb V$ as defining a grading $\Bbb V=\Bbb
V_0\oplus\dots\oplus\Bbb V_N$ for some $N\in\Bbb N$ such that $\frak
g_i\cdot\Bbb V_j\subset\Bbb V_{i+j}$. Via forming direct sums and
complexifications this extends to general completely reducible
representations of $\frak p$.

Now we can combine this with the grading on $\frak q_+/\frak p_+$
induced by the grading of $\frak g$ coming from $\frak q$, to obtain
$\frak q_0$--invariant gradings on all the chain spaces $C_*(\frak
q_+/\frak p_+,\Bbb V)$. The isomorphism $(\frak q_+/\frak
p_+)^*\cong\frak p_0\cap\frak q_-$ is compatible with the gradings and
in the picture of multilinear maps from $\frak p_0\cap\frak q_-$ to
$\Bbb V$ the grading on $C_*(\frak q_+/\frak p_+,\Bbb V)$ is given by
the usual homogeneity of multilinear maps between graded vector spaces.

While these gradings are not $\frak q$--invariant, the fact that
$\frak q=\oplus_{i\geq 0}\frak g_i$ immediately implies that the
filtrations by right ends induced by these gradings all are $\frak
q$--invariant. Denoting the grading components by $C_*(\frak q_+/\frak
p_+,\Bbb V)_i$, the corresponding filtration is defined by 
$$
C_*(\frak q_+/\frak p_+,\Bbb V)^j:=\oplus_{i\geq j}C_*(\frak q_+/\frak
p_+,\Bbb V)_i. 
$$
Given a filtered $\frak q$--module, we can pass to the associated
graded module, which by definition is just the direct sum of the
quotients of each filtration component by the next smaller one. From
the construction, it is clear that as a $\frak q_0$--module, the
associated graded $\gr(C_*(\frak q_+/\frak p_+,\Bbb V))$ is isomorphic
to $C_*(\frak q_+/\frak p_+,\Bbb V)$. But since by construction $\frak
q_+$ maps any filtration component to the next smaller one, $\frak
q_+$ acts trivially on $\gr(C_*(\frak q_+/\frak p_+,\Bbb V))$, so this
is a completely reducible representation of $\frak q$. 

In general, there is neither a natural map from a filtered module to
its associated graded nor a natural map in the other
direction. However, the grading we start with, defines such a
mapping. Explicitly, given $\ph\in C_*(\frak q_+/\frak p_+,\Bbb V)$,
we can uniquely write $\ph=\sum\ph_i$ with $\ph_i\in C_*(\frak
q_+/\frak p_+,\Bbb V)_i$. Denoting by 
$$
q_i:C_*(\frak q_+/\frak p_+,\Bbb V)^i\to\gr_i(C_*(\frak q_+/\frak
p_+,\Bbb V)) 
$$
the canonical projection, our isomorphism is defined by mapping $\ph$
to $\sum_iq_i(\ph_i)$. In this way, we can interpret the maps
$\partial^*_\rho$, $\partial_\rho$ and $\square_\rho$ as $\frak
q$--homomorphisms defined on $\gr(C_*(\frak q_+/\frak p_+,\Bbb V))$
and we then have the Hodge decomposition on this associated graded.

On the other hand, $\partial^*_\rho$ can also be considered as a
$\frak q$--homomorphism defined on $C_*(\frak q_+/\frak p_+,\Bbb V)$
itself. From the explicit formula in Section \ref{2.2} it follows
immediately that $\partial^*_\rho$ actually respects the grading on
the chain space. This implies that for $\ph\in C_*(\frak q_+/\frak
p_+,\Bbb V)^j$ we obtain $\partial^*_\rho\ph\in C_{*-1}(\frak q_+/\frak
p_+,\Bbb V)^j$ and
$q_j(\partial^*_\rho\ph)=\partial^*_\rho(q_j(\ph))$, where the
$\partial^*_\rho$ on the right hand side is the one on the associated
graded. This justifies denoting both maps by the same symbol.

\subsection{The action of $\square_\rho$}\label{2.6}
The next step towards a relative version of Kostant's theorem is an
explicit description of the action of $\square_\rho$ on $C_*(\frak
q_+/\frak p_+,\Bbb V)$. As a Lie algebra and as a $\frak q_0$--module,
we can identify $\frak q_+/\frak p_+$ with $\frak q_+\cap\frak p_0$,
compare with Section \ref{2.5}. Hence we get an isomorphism $C_k(\frak
q_+/\frak p_+,\Bbb V)\cong\La^k(\frak q_+\cap\frak p_0)\otimes\Bbb V$,
and the inclusion of the subalgebra $\frak q_+\cap\frak p_0$ into
$\frak p_0$ induces an inclusion
$$
j:C_k(\frak q_+\cap\frak
p_0,\Bbb V)\to C_k(\frak p_0,\Bbb V).
$$
This is compatible with the Lie algebra homology differentials, which
we therefore both denote by $\partial^*$.

The Killing form $B$ of $\frak g$ has non--degenerate restrictions to
$\frak p_0$ and $\frak q_0$, so from the decomposition
\eqref{g-decomp}, we see that it induces an isomorphism $\frak
p_0\cong\frak p_0^*$ which restricts to a duality between $\frak
q_+\cap\frak p_0$ and $\frak q_-\cap\frak p_0$. This restriction is
exactly the duality we used in Section \ref{2.3} to define
$\partial_\rho$ and of course, the full duality can be used to define
a Lie algebra cohomology differential $\partial_{\frak p_0}$ on
$C_*(\frak p_0,\Bbb V)$.

Having this at hand, the computations in sections 3.3.2 and 3.3.3 of
\cite{book} can be used without any change in our situation. To
formulate the result, we observe that any element $X\in\frak p_0$
naturally acts on $\Bbb V$. Likewise, $X$ acts on $\frak p_0$ by the
adjoint action and this induces an action of $\La^*\frak
p_0$. Together, these actions determine an action on $C_*(\frak
p_0,\Bbb V)$, and we write $\Cal L_X$ for the action of $X$. We can
decompose $\Cal L_X=\Cal L^{\frak p_0}_X+\Cal L^{\Bbb V}_X$ where in
the first part $X$ acts on $\La^*\frak p_0$ only, while in the second
part it acts only on $\Bbb V$. Using this, we can formulate the result
as follows.

\begin{prop}\label{prop2.6}
  Let $\{\xi_\ell\}$ be a basis for $\frak p_0$, which is the union of
  a basis of $\frak q_-\cap\frak p_0$ and of a basis of $\frak
  q\cap\frak p_0$ and let $\{\eta_\ell\}$ be the dual basis with
  respect to $B$. Let $j:C_*(\frak q_+\cap\frak p_0,\Bbb V)\to
  C_*(\frak p_0,\Bbb V)$ be the inclusion. Then we have
$$
j\o\square_\rho=\tfrac12\left(-\textstyle\sum_\ell\Cal L^{\Bbb V}_{\eta_\ell}\Cal
L^{\Bbb V}_{\xi_\ell}-\sum_{\ell:\xi_\ell\in\frak q_-}\Cal L_{\eta_\ell}\Cal
L_{\xi_\ell}+\sum_{\ell:\xi_\ell\in\frak q}\Cal L_{\eta_\ell}\Cal
L_{\xi_\ell}\right)\o j.
$$
\end{prop}

\subsection{The action on isotypical components}\label{2.7} 
To continue, we restrict to the complex case, i.e.~we assume that
$\frak q\subset\frak p\subset\frak g$ are nested parabolic subalgebras
in a complex semisimple Lie algebra $\frak g$ and that $\Bbb V$ is a
complex irreducible representation of $\frak p_0$. This means that
$\Bbb V$ is a complex irreducible representation of the semisimple part
$\frak p_0^s$ of $\frak p_0$ on which the center $\frak z(\frak p_0)$
acts diagonalizably.

Usually, one describes such representations by highest weights, but in
our setting it is better to use the negatives of lowest weights.  Such
a weight is a complex linear functional $\la$ on the Cartan subalgebra
$\frak h$ of $\frak g$. (Recall that $\frak h$ naturally decomposes as
the direct sum of the center $\frak z(\frak p_0)$ and a Cartan
subalgebra for $\frak p_0^s$.) The functionals occurring in this way
are exactly the $\frak p$--algebraically integral ones which in
addition are $\frak p$--dominant, which means that
$\langle\la,\al\rangle$ is a non--negative integer for all positive
roots $\al$ such that $\frak g_\al\subset\frak p_0$.

As we have noted already, we can view the chain spaces as $C_*(\frak
p_0\cap\frak q_+,V)$ and they are completely reducible $\frak
q_0$--modules. Now irreducible representations of $\frak q_0$ are also
determined by the negatives of their lowest weights which again are
linear functionals on $\frak h$. Here, they have to be $\frak
q$--algebraically integral and $\frak q$--dominant, so the condition
that $\langle\la,\al\rangle$ is a non--negative integer is only
required if $\frak g_\al\subset\frak q_0$. In particular, for a $\frak
q$--dominant integral weight $\nu$, there is the $\frak
q_0$--isotypical component $W^\nu\subset C_*(\frak p_0\cap\frak
q_+,V)$ of lowest weight $-\nu$, which is the $\frak q_0$--submodule
generated by all $\frak q_0$--lowest weight vectors of that weight.

Now we can compute the action of $\square_\rho$ on any isotypical
component. To formulate the result, we denote by $\delta_{\frak p}$
the \textit{lowest form} of the semisimple Lie algebra $\frak p_0^s$,
i.e.~half of the sum of its positive roots.

\begin{cor}\label{cor2.7}
Let $\frak q\subset\frak p$ be nested standard parabolic subalgebras
in a complex semisimple Lie algebra $\frak g$ and let $\Bbb V$ be a
complex irreducible representation of $\frak p$ with lowest weight
$-\la\in\frak h^*$. Let $W^{\nu}\subset C_*(\frak p_0\cap\frak
q_-,\Bbb V)$ be the $\frak q_0$--isotypical component of lowest weight
$-\nu$.

Then $\square_\rho$ acts on $W^{\nu}$ by multiplication by the scalar
$\frac12(\|\la+\delta_{\frak p}\|^2-\|\nu+\delta_{\frak p}\|^2)$,
where the norm is induced by the Killing form of $\frak g$.
\end{cor}
\begin{proof}
To prove the result, it suffices to show that $\square_ \rho$ acts on
a lowest weight vector of weight $-\nu$ by multiplication with the
scalar in question. We use the formula for $\square_\rho$ from
Proposition \ref{prop2.6} with respect to appropriately chosen dual
bases. Recall that the Cartan subalgebra $\frak h$ of $\frak g$ is
contained in $\frak p_0$ and that $B$ is positive definite on the
subspace $\frak h_0\subset\frak h$ on which all roots (of $\frak g$)
are real. We start by choosing an orthonormal basis
$\{H_1,\dots,H_r\}$ for $\frak h_0$ which then is a complex basis
for $\frak h$. Next, let $\Delta^+(\frak p_0)$ be the set of those
positive roots $\al$ of $\frak g$ for which the root space $\frak
g_\al$ is contained in $\frak p_0$. For each $\al\in\Delta^+(\frak
p_0)$, we choose elements $E_\al\in\frak g_\al$ and $F_\al\in\frak
g_{-\al}$ such that $B(E_\al,F_\al)=1$. Then $\{E_\al,H_i,F_\al\}$ is
a basis for $\frak p_0$ whose dual basis with respect to $B$ is given
by $\{F_\al,H_i,E_\al\}$. Moreover, for each $\al$ we see that
$[E_\al,F_\al]$ is dual to $\al$ with respect to $B$.

Having chosen these bases, one completes the proof as for Proposition
3.3.4 in \cite{book}, using the evident decomposition $\Delta^+(\frak
p_0)=\Delta^+(\frak q_0)\sqcup\Delta^+(\frak q_+\cap\frak p_0)$
according to location of the root spaces.
\end{proof}

\subsection{The relative Hasse diagram}\label{2.8}
The statement of Kostant's theorem is based on a subset in the Weyl
group $W$ of $\frak g$, the so--called Hasse diagram associated to a
parabolic subalgebra. We next introduce a relative version of
this. Recall that for a parabolic subalgebra, one defines the Weyl
group as the Weyl group of the semisimple part of a Levi factor, which
is naturally a subgroup of $W$ and thus acts on $\frak h^*$. In our
situation of two nested parabolics $\frak q\subset\frak p\subset\frak
g$ we thus have $W_{\frak q}\subset W_{\frak p}\subset W$.

Recall further that denoting by $\Delta^+$ the set of positive roots
of $\frak g$, one associates to $w\in W$ the subset
$\Ph_w:=\{\al\in\Delta^+:w^{-1}(\al)\in-\Delta^+\}\subset\Delta^+$
which uniquely determines $w$. Using a notation based on the location
of root spaces as in Section \ref{2.7} we can write
$\Delta^+=\Delta^+(\frak p_0)\sqcup\Delta^+(\frak p_+)$ and we can
further decompose $\Delta^+(\frak p_0)=\Delta^+(\frak
q_0)\sqcup\Delta^+(\frak p_0\cap\frak q_+)$. It is well known (see
section 3.2.15 of \cite{book}) that $w\in W_{\frak p}$ if and only if
$\Ph_w\subset\Delta^+(\frak p_0)$ and likewise for $W_{\frak q}$. The
\textit{Hasse diagram} $W^{\frak p}$ associated to the parabolic
subalgebra $\frak p$ is defined as the set of those $w\in W$ for which
$\Ph_w\subset\Delta^+(\frak p_+)$, and likewise for $W^{\frak q}$.

\begin{definition}\label{def2.8}
  For two nested parabolic subalgebras $\frak q\subset\frak p$ in a
  complex semisimple Lie algebra $\frak g$, we define the
  \textit{relative Hasse diagram} $W^{\frak q}_{\frak p}\subset W$ by 
$$
W^{\frak q}_{\frak
  p}:=\{w\in W:\Ph_w\subset\Delta^+(\frak p_0\cap\frak q_+)\}
$$
\end{definition}

From the above discussion, we conclude that $W^{\frak q}_{\frak p}$
coincides with the intersection $W^{\frak q}\cap W_{\frak p}$ of the
Hasse diagram of $\frak q$ with the Weyl group of $\frak p$. On the
other hand, as a set $W^{\frak q}_{\frak p}$ can be identified with
the Hasse diagram of the parabolic subalgebra $(\frak q\cap\frak
p_0^s)\subset\frak p_0^s$. The definition we have chosen has the
advantage that $W^{\frak q}_{\frak p}$ naturally acts on all of $\frak
h^*$.

Relative Hasse diagrams can be determined in a similar way as usual
ones and some of the basic properties carry over to the relative case.
\begin{lemma}\label{lem2.8}
  Let $\frak q\subset\frak p\subset\frak g$ be two nested parabolic
  subalgebras in a complex semisimple Lie algebra and let $W$ be the
  Weyl group of $\frak g$. 

  (1) Let $\delta^{\frak q}_{\frak p}$ denote the sum of all fundamental
  weights corresponding to simple roots contained in $\Delta^+(\frak
  p_0\cap\frak q_+)$. Then the map $w\mapsto w^{-1}(\delta^{\frak
    q}_{\frak p})$ restricts to a bijection between $W^{\frak
    q}_{\frak p}$ and the orbit of $\delta^{\frak q}_{\frak p}$ under
  $W_{\frak p}$.

(2) For any $\frak p$--dominant weight $\la$ and any element $w\in
W^{\frak q}_{\frak p}\subset W$, the weight $w(\la)$ is $\frak
q$--dominant.
\end{lemma}
\begin{proof}
  (1) By construction, $\delta^{\frak q}_{\frak p}$ is orthogonal to each
  simple root $\al\in\Delta^+(\frak q_0)$ and thus it is stabilized by
  the reflection corresponding to such a root. Since these reflections
  generate $W_{\frak q}$, we see that $w(\delta^{\frak q}_{\frak
    p})=\delta^{\frak q}_{\frak p}$ for any $w\in W_{\frak
    q}\subset W_{\frak p}$. 

Conversely, if $w\in W$ satisfies $w(\delta^{\frak q}_{\frak
  p})=\delta^{\frak q}_{\frak p}$, then for a root $\al$, we have
$$
\langle \delta^{\frak q}_{\frak p},w^{-1}(\al)\rangle=\langle
w(\delta^{\frak q}_{\frak p}),\al\rangle=\langle \delta^{\frak q}_{\frak
  p},\al\rangle,
$$
so this is positive for each $\al\in\Delta^+(\frak q_+)$. But this
implies that $\Ph_w\subset\Delta^+(\frak q_0)$ and hence $w\in
W_{\frak q}$. 

  Applying Proposition 5.13 of \cite{Kostant} (Proposition 3.2.15 in
  \cite{book}) to the parabolic subalgebra $\frak q\cap\frak
  p_0^s\subset\frak p_0^s$, we conclude that any element $w\in W_{\frak
    p}$ can be uniquely written as $w=w_1w_2$ with $w_1\in W_{\frak
    q}$ and $w_2\in W^{\frak q}_{\frak p}$. Hence $w^{-1}(\delta^{\frak
    q}_{\frak p})=w_2^{-1}(\delta^{\frak q}_{\frak p})$, and we conclude
  that $w\mapsto w^{-1}(\delta^{\frak q}_{\frak p})$ defines a surjection
  from $W^{\frak q}_{\frak p}$ onto the $W_{\frak p}$--orbit of
  $\delta^{\frak q}_{\frak p}$. But if $w,\tilde w\in W^{\frak q}_{\frak
    p}$ satisfy $w^{-1}(\delta^{\frak q}_{\frak p})=\tilde
  w^{-1}(\delta^{\frak q}_{\frak p})$, then $w\tilde w^{-1}$ fixes
  $\delta^{\frak q}_{\frak p}$ and hence lies in $W_{\frak q}$. If $w$
  and $\tilde w$ were different, then $w=(w\tilde w^{-1})\tilde w$
  would contradict uniqueness of the product decomposition. 

\medskip

(2) For $\al\in\Delta^+(\frak q_0)$ and $w\in W^{\frak q}_{\frak p}$, we
by definition know that $w^{-1}(\al)\in\Delta^+(\frak p)$. Hence for a
a $\frak p$--dominant weight $\la$ we have 
$$
0\leq \langle\la,w^{-1}(\al)\rangle=\langle w(\la),\al\rangle,
$$
so $w(\la)$ is $\frak q$--dominant. 
\end{proof}

We will give an example for determining $W^{\frak q}_{\frak p}$ in
Example \ref{ex3.2} below.

\subsection{The relative version of Kostant's theorem}\label{2.9}  
The last ingredient needed to formulate the first main result of this
article is the affine action of the Weyl group on weights. For a
weight $\la$ and $w\in W$, we define $w\cdot\la=w(\la+\delta)-\delta$,
where $\delta=\frac12\sum_{\al\in\Delta^+}\al$ is the lowest form of
the Lie algebra $\frak g$. This has the nice property that $w\cdot
0=w(\delta)-\delta=-\sum_{\al\in\Ph_w}\al$, see Proposition 3.2.14 in
\cite{book}. From the proof of this Proposition and the fact that an
element $w\in W_{\frak p}$ has to preserve the set $\Delta^+(\frak
p_+)$ one easily concludes that $w\cdot 0=w(\delta_{\frak
  p})-\delta_{\frak p}$ for $w\in W_{\frak p}$. In particular, for
$w\in W^{\frak q}_{\frak p}$, we obtain $w\cdot\la=w(\la+\delta_{\frak
  p})-\delta_{\frak p}$.  If $\la$ is $\frak p$--dominant, then for
any $\al\in\Delta^+(\frak p_0)$ we have $\langle\la+\delta_{\frak
  p},\al\rangle\geq 1$ and using this, one concludes as in the proof
of Lemma \ref{lem2.8} that $w\cdot\la$ is always a $\frak q$--dominant
weight. Using this, we can now formulate:

\begin{thm}\label{thm2.9}[Relative version of Kostant's theorem] Consider two
  nested complex parabolic subalgebras $\frak q\subset\frak p$ in a
  complex semisimple Lie algebra $\frak g$. Then for a finite
  dimensional complex irreducible representation $\Bbb V$ of $\frak p$
  with lowest weight $-\la\in\frak h^*$, the homology space $H_*(\frak
  q_+/\frak p_+,\Bbb V)$ is a completely reducible representation of
  $\frak q_0$ with the following structure:

For a $\frak q$--dominant weight $\nu$, the isotypical component
$H_*(\frak q_+/\frak p_+,\Bbb V)^\nu$ of lowest weight $-\nu$ is
non--zero if and only if $\nu=w\cdot\la$ for some $w\in W^{\frak
  q}_{\frak p}$. If this is the case, then the isotypical component is
irreducible and contained in $H_{\ell(w)}(\frak q_+/\frak p_+,\Bbb
V)$, where $\ell(w)$ is the length of $w$. The weight $-w\cdot\la$
even occurs with multiplicity one in $\La^*(\frak q_+/\frak
p_+)\otimes\Bbb V$.
\end{thm}
\begin{proof}
We have already noted that the weight $\nu_w:=w\cdot\la$ is $\frak
q$--dominant. We have also seen above that
$\nu_w=w(\la)-\sum_{\al\in\Ph_w}\al$. Now since $-\la$ is the lowest
weight of the irreducible representation $\Bbb V$ of $\frak p_0$, also
$-w(\la)$ is a weight of $\Bbb V$. On the other hand, since $w\in
W^{\frak q}_{\frak p}$, we have $\Ph_w\subset\Delta^+(\frak
p_0\cap\frak q_+)$. This exactly means that each $\al\in\Ph_w$ is a
weight of $\frak q_+/\frak p_+$ and thus $\sum_{\al\in\Ph_w}\al$ is a
weight of $\La^k(\frak q_+/\frak p_+)$, where $k=|\Ph_w|$. It is well
known that $|\Ph_w|$ coincides with the length $\ell(w)$. Hence we
have verified that $-\nu_w$ indeed is a weight of the $\frak
q_0$--representation $\La^{\ell(w)}(\frak q_+/\frak p_+)\otimes\Bbb
V$. On the other hand, since $\nu_w+\delta_{\frak
  p}=w(\la+\delta_{\frak p})$ we see that $\|\nu_w+\delta_{\frak
  p}\|=\|\la+\delta_{\frak p}\|$. Hence by Corollary \ref{cor2.7}, if
$-\nu_w$ actually is a lowest weight in the representation
$\La^{\ell(w)}(\frak q_+/\frak p_+)\otimes\Bbb V$ then its isotypical
component will be contained in $\ker(\square_\rho)$ which is
isomorphic to the homology by Lemma \ref{lem2.4}. 

The main step to complete the proof now is to derive an analog of
Lemma 5.12 of \cite{Kostant} (Lemma 3.3.5 of \cite{book}). This is
rather straightforward along the lines of Cartier's simplified
argument from \cite{Cartier}, so we just sketch it. A weight $-\nu$ of
$\La^*(\frak q_+/\frak p_+)\otimes\Bbb V$ can be written as
\begin{equation}\label{nudecomp}
-\mu+\textstyle\sum_{\al\in\Ps}\al,
\end{equation}
where $-\mu$ is a weight of $\Bbb V$ and $\Ps$ is some subset of
$\Delta^+(\frak p_0\cap\frak q_+)$. Moreover, the multiplicity of
$-\nu$ as a weight coincides with the sum of the multiplicities of the
weights $-\mu$ in $\Bbb V$ over all decompositions of $-\nu$ as in
\eqref{nudecomp}. 

Fixing a weight $-\nu$ decomposed in this way, there is an element
$w\in W_{\frak p}$ such that $w(-\nu-\delta_{\frak p})$ is $\frak
p$--dominant, and of course
$$
w(-\nu-\delta_{\frak p})=w(-\mu)-w(\delta_{\frak p}-\textstyle\sum_{\al\in\Ps}\al).
$$ 
Now $w(-\mu)$ is a weight of $\Bbb V$ and thus can be obtained from
$-\la$ by adding a linear combination of simple roots from
$\Delta^+(\frak p_0)$ with non--negative integral coefficients.  As in
the proof of Lemma 3.3.5 of \cite{book}, one next shows that
$w(\delta_{\frak p}-\sum_{\al\in\Ps}\al)$ is obtained by subtracting
from $\delta_{\frak p}$ a linear combination of some simple roots from
$\Delta^+(\frak p_0)$ with non--negative integral
coefficients. Altogether, we see that $\la+\delta_{\frak
  p}=w(\nu+\delta_{\frak p})+\sum n_i\al_i$ for some non--negative
integers $n_i$ and simple roots $\al_i\in\Delta^+(\frak p_0)$.  Using
that $w(\nu+\delta_{\frak p})$ is $\frak p$--dominant, this easily
implies that
$$
\|\la+\delta_{\frak p}\|\geq \|w(\nu+\delta_{\frak p})\|=\|\nu+\delta_{\frak p}\|
$$ with equality if and only if all $n_i$ are zero. The latter
condition means that $w(\mu)=\la$ and that $w(\delta_{\frak
  p}-\sum_{\al\in\Ps}\al)=\delta_{\frak p}$. Hence we obtain
$\mu=w^{-1}(\la)$ and $w^{-1}(\delta_{\frak p})=\delta_{\frak
  p}-\sum_{\al\in\Ps}\al$. The last condition implies that
$\Ps=\Ph_{w^{-1}}$, which shows that $\nu=\nu_{w^{-1}}$, $w^{-1}\in
W^{\frak q}_{\frak p}$, and that there is only one possible
decomposition of $\nu$ as in \eqref{nudecomp}.

Using this uniqueness, multiplicity one of the weight $-w^{-1}(\la)$ in
$\Bbb V$ implies that $\nu_{w^{-1}}$ has multiplicity one as a weight
of $C_*(\frak q_+/\frak p_+,\Bbb V)$. Finally, since
$\la+\delta_{\frak p}$ lies in the interior of the dominant Weyl
chamber for $\frak p$, it follows that for $w\neq w'\in W^{\frak
  q}_{\frak p}$ we get $\nu_w\neq\nu_{w'}$.

In view of Corollary \ref{cor2.7}, we can complete the proof by showing
that each $-\nu_w$ for $w\in W^{\frak q}_{\frak p}$ is actually a
lowest weight of $C_*(\frak q_+/\frak p_+,\Bbb V)$. But this follows
as for the absolute version of Kostant's theorem, by showing that for
$\al\in\Delta^+(\frak q_0)$ one has $\|\nu_w+\al+\delta_{\frak
  p}\|>\|\la+\delta_{\frak p}\|$, so $-\nu_w-\al$ cannot be a weight
of $C_*(\frak q_+/\frak p_+,\Bbb V)$.
\end{proof}

\section{Relative and absolute Homology}\label{3}

We start by describing the relation between relative and absolute
Hasse dia\-grams. Since irreducible representations of $\frak p$ can
have singular infinitesimal character, relative homology groups
realize parts of an affine Weyl--orbits in either regular or
singular infinitesimal character. In the case of regular infinitesimal
character, we show that each affine Weyl orbit decomposes into a
disjoint union of sequences of relative homology groups, and show how
to obtain the individual sequences in terms of $\frak q$--invariant
data.

\subsection{Lie algebra homology and affine Weyl orbits}\label{3.1} 
The reason for the importance of Kostant's theorem in the theory of
parabolic geometries is its relation to infinitesimal character. Given
a parabolic subalgebra $\frak q$ in a semisimple Lie algebra $\frak g$
and corresponding groups $Q\subset G$, there is an associated
geometric structure. These so--called parabolic geometries can be
uniformly described in terms of Cartan connections, see the exposition
in \cite{book}. The class of parabolic geometries contains important
examples like conformal structures and CR--structures and has been
intensively studied during the last years. Via a construction of
associated bundles, any representation of $Q$ determines a natural
vector bundle on any manifold endowed with a parabolic geometry of type
$(G,Q)$. One of the difficult and important questions then is to
describe differential operators acting on sections of such bundles,
which are intrinsic to the geometry in question. 

This has a close connection to representation theory. The
\textit{homogeneous model} of parabolic geometries of type $(G,Q)$ is
the homogeneous space $G/Q$. For this example, the natural vector
bundles as described above are exactly the homogeneous vector
bundles, and differential operators intrinsic to the geometry are
exactly those which are intertwining operators for that natural
$G$--representations on spaces of smooth sections of homogeneous
vector bundles. Via a duality, intertwining operators which are
differential operators are related to homomorphisms of induced
modules, see Section 1.4.10 in \cite{book}. 

If one considers homogeneous bundles associated to irreducible
representations of $Q$, then the resulting induced modules are
generalized Verma modules. These are modules having an infinitesimal
character, and clearly a non--zero homomorphism between two such
modules can only exist if their infinitesimal characters agree. By a
classical theorem of Harish--Chandra, this is true if and only if the
highest weights of the inducing representations can be obtained from
each other by the affine action of an element of the Weyl group. (The
duality mentioned above is the reason why in the context related to
differential operators it is more natural to work with negatives of
lowest weights rather than highest weights.) This shows that invariant
differential operators for parabolic geometries are rather rare, which
is one of the reasons why they are interesting.

While the affine action of the Weyl group is easy to understand in
terms of weights, it is a priori not at all clear (in particular in a
geometric picture) how the representations in an affine Weyl orbit
(which form the candidates for inducing domains and targets for
invariant differential operators) are related. For the further
discussion, we have to distinguish between affine Weyl orbits of
\textit{regular} and of \textit{singular} infinitesimal
character. Regular orbits are those involving points in the interiors
of Weyl chambers, while singular orbits are contained in walls.

In the case of regular infinitesimal character (and of integral
weights) Kostant's theorem provides a satisfactory solution to this
problem. Any regular orbit contains a weight which lies in the
interior of the dominant Weyl chamber, so it is of the form
$\la+\delta$ for a $\frak g$--dominant weight $\la$. If, in addition,
$\la$ is integral, then there is a finite--dimensional irreducible
representation $\Bbb V$ of $\frak g$ corresponding to $\la$. Kostant's
theorem then implies that the completely reducible representation
$H_*(\frak q_+,\Bbb V)$ is the direct sum of one copy of each of the
irreducible $\frak q$--representations with highest weights contained
in the affine Weyl--orbit of $\la$. The $\frak q$--invariant
description in terms of the standard complex computing Lie algebra
homology can then be directly transferred to geometry and this is the
basis for the construction of BGG--sequences in \cite{CSS-BGG}.

The relative Lie algebra homology groups studied in Section \ref{2}
can serve a similar purpose in some cases of singular infinitesimal
character, see Example \ref{ex3.2}. On the other hand, in regular
infinitesimal character, one can nicely relate absolute homology
groups to relative ones. In both situations this has important
consequences for invariant differential operators, which are studied
in \cite{part2}.

\subsection{The relation between absolute and relative Hasse
  diagrams}\label{3.2} 
Let us return to our standard setting of two nested parabolic
subalgebras $\frak q\subset\frak p$ in a complex semisimple Lie
algebra $\frak g$. Then we have the two subgroups $W_{\frak q}\subset
W_{\frak p}$, the two (absolute) Hasse diagrams $W^{\frak q}\supset
W^{\frak p}$, and the relative Hasse diagram $W^{\frak q}_{\frak p}$
from Definition \ref{def2.8}. All these subsets admit a nice
description in terms of the set $\Ph_w\subset\Delta^+$ associated to a
Weyl group element $w\in W$, compare with Section \ref{2.8}. The basis
for the further discussion is the following simple result.

\begin{prop}\label{prop3.2}
  Let $\frak q\subset\frak p$ be two nested parabolic subalgebras in a
  complex semisimple Lie algebra $\frak g$. Then multiplication in the
  Weyl group $W$ of $\frak g$ induces a bijection
$$
W^{\frak q}_{\frak p}\x W^{\frak p}\to W^{\frak q}.
$$
Moreover, for $w_1\in W^{\frak q}_{\frak p}$ and $w_2\in W^{\frak p}$,
we get $\ell(w_1w_2)=\ell(w_1)+\ell(w_2)$.
\end{prop}
\begin{proof}
  Take $w_1\in W^{\frak q}_{\frak p}$ and $w_2\in W^{\frak p}$, so we
  know that $\Ph_{w_1}\subset\Delta^+(\frak p_0\cap\frak q_+)$ and
  $\Ph_{w_2}\subset\Delta^+(\frak q_+)$, and put $w=w_1w_2$. For a
  positive root $\al\in\Delta^+(\frak q_0)$, we thus get
  $w_1^{-1}(\al)\in\Delta^+$. Since $\frak q_0\subset\frak p_0$ and
  $w_1\in W_{\frak p}$, we even get $w_1^{-1}(\al)\in\Delta^+(\frak
  p_0)$. But this implies $w_2^{-1}(w_1^{-1}(\al))\in\Delta^+$. Hence
  we see that $\Ph_w\subset\Delta^+(\frak q_+)$, so $w\in W^{\frak
    q}$. 

  Conversely, Proposition 5.13 of \cite{Kostant} (Proposition 3.2.15
  in \cite{book}) for the parabolic subalgebra $\frak p\subset\frak g$
  says that each element $w\in W$ can be uniquely written as
  $w=w_1w_2$ with $w_1\in W_{\frak p}$ and $w_2\in W^{\frak p}$ and
  that $\ell(w)=\ell(w_1)+\ell(w_2)$. Hence we can complete the proof
  by showing that $w\in W^{\frak q}$ implies $w_1\in W^{\frak
    q}_{\frak p}$. But this follows immediately from the proof of this
  result in \cite{book}, since there $w_1$ is obtained as the unique
  element of $W_{\frak p}$ such that
  $\Ph_{w_1}=\Ph_w\cap\Delta^+(\frak p_0)\subset\Delta^+(\frak
  p_0\cap\frak q_+)$.
\end{proof}

This result significantly simplifies the determination of the Hasse
diagrams of non--maximal parabolics as well as the affine orbits of
weights under this Hasse diagram. Let us describe this in a simple
example, in which the Hasse diagram is available in Section 3.2.16 of
\cite{book}. 

\begin{example}\label{ex3.2}
Consider $\frak g=\frak{sl}(4,\Bbb C)$, let $\frak p$ be the maximal
parabolic subalgebra corresponding to the first simple root and let
$\frak q$ be the parabolic subalgebra corresponding to the first two
simple roots. In the language of crossed Dynkin diagrams, these are 
$$
\frak p=\xdd{}{}{} \qquad \frak q=\xxd{}{}{} . 
$$  
(1) Denoting the reflections corresponding to the simple roots by
$\si_i$, $i=1,2,3$ with roots numbered from left to right, it is clear
by definition that $W_{\frak p}$ is generated by $\si_2$ and
$\si_3$. Moreover, $\delta^{\frak q}_{\frak p}$ is just the second
fundamental weight in this case. To represent a weight, we write it as
a linear combination of fundamental weights and then write the
coefficient over the vertex for the corresponding simple root in the
crossed Dynkin diagram. In this language, it is easy to compute the
action of simple reflections on weights, see Section 3.2.16 of
\cite{book}. This shows that the $W_{\frak p}$--orbit of
$\delta^{\frak q}_{\frak p}$ is
$$
\xdd{0}{1}{0} \to \xdd{1}{-1}{1} \to \xdd{1}{0}{-1}
$$
and $W^{\frak q}_{\frak p}=\{e,\si_2,\si_2\si_3\}$ with the elements
of length $0$, $1$, and $2$. This immediately allows us to compute the
affine $W^{\frak q}_{\frak p}$--orbit of a general weight as 
\begin{equation}
  \label{eq:wqp-orb}
\xxd{a}{b}{c} \to \qquad \wxxd{a+b+1}{-b-2}{b+c+1} \qquad \to \qquad
\Wxxd{a+b+c+2}{-b-c-3}{b}    
\end{equation}

Now the basic case of interest for computing homology is that the
initial weight is $\frak p$--dominant and integral, i.e.~that $a$, $b$,
and $c$ are integers with $b,c\geq 0$. Then the three weights in the
above pattern, which visibly are $\frak q$--dominant and integral, are
the negatives of the lowest weights of the irreducible representations
$H_i(\frak q_+/\frak p_+,\Bbb V)$ for $i=0,1,2$. Here $\Bbb V$ is the
irreducible representation of $\frak p$ with the negative of the
lowest weight equal to $\xdd{a}{b}{c}$. 

(2) Now we can use Proposition \ref{prop3.2} to determine the
Hasse diagram of $\frak q$ as well as affine Weyl orbits. Similarly
as in (1), one can determine the Weyl orbit of $\ddd{1}{0}{0}$ to
obtain that $W^{\frak p}=\{e,\si_1,\si_1\si_2,\si_1\si_2\si_3\}$ with
elements of length $0$, $1$, $2$, and $3$. Together with the
description of $W^{\frak q}_{\frak p}$ from part (1), this immediately
gives the 12--element set $W^{\frak q}$. Likewise, we can easily
determine the affine orbit of a general weight under $W^{\frak p}$ as 
\begin{equation}
  \label{eq:wp-orb}
\xdd{a}{b}{c} \to\quad \awxdd{-a-2}{a+b+1}{c} \to \quad
\bwxdd{-a-b-3}{a}{b+c+1}\qquad \to \qquad \wxdd{-a-b-c-4}{a}{b}
\end{equation}
The full $W^{\frak q}$--orbit is then obtained by shifting each of
these weights according to \eqref{eq:wqp-orb}. 

\medskip

In the case of integral weights of regular infinitesimal character,
the initial weight in \eqref{eq:wp-orb} is $\frak g$--dominant and
integral, and we see that the $W^{\frak q}$--orbit of the weight is
the disjoint union of four $W^{\frak q}_{\frak p}$--orbits as in
\eqref{eq:wqp-orb} with the initial weights coming from 
\eqref{eq:wp-orb}.

The situation in singular infinitesimal character is more subtle. The
largest singular orbits (for which each element lies in one wall but
not in the intersection of two walls) are obtained by setting one
coefficient in the initial weight equal to $-1$. In each of the three
possible cases, one verifies that the $W^{\frak p}$--orbit of the
corresponding weight degenerates to a three--element set, with one of
the elements $\frak p$--dominant in each case. Using
\eqref{eq:wqp-orb} to determine $W^{\frak q}_{\frak p}$--orbits, we
obtain three basic patterns in singular infinitesimal character,
namely
\begin{gather*}
\xxd{-1}{a}{b} \ \to \ \waxxd{a}{-a-2}{a+b+1} \quad \to \quad
\Wxxd{a+b+1}{-a-b-3}{a} \\
\wcxxd{-a-2}{a}{b} \ \to \ \waxxd{-1}{-a-2}{a+b+1} \quad \to \
\wxxd{b}{-a-b-3}{a} \\
\wbxxd{-a-b-3}{a}{b} \ \to \quad \wxxd{-b-2}{-a-2}{a+b+1} \quad \to \ 
\wxxd{-1}{-a-b-3}{a} 
\end{gather*}
with $a,b\geq 0$. For all these cases, the construction in
\cite{part2} produces invariant differential operators in singular
infinitesimal character.

To conclude this example, we remark that there is another maximal
parabolic $\tilde{\frak p}$ containing $\frak q$, which corresponds to
the crossed Dynkin diagram $\dxd$. For this case, the relative Hasse
diagram $W^{\frak q}_{\tilde{\frak p}}$ is the two--element set
$\{e,\si_1\}$ while $W^{\tilde{\frak p}}$ has six elements, compare
with Example 3.2.17 in \cite{book}. Thus one obtains a decomposition
of $W^{\frak q}$ into two copies of $W^{\tilde{\frak p}}$. Both this
decomposition and the one into three copies of $W^{\frak p}$ can be
spotted in the picture for $W^{\frak q}$ on top of p.~330 of
\cite{book}. 
\end{example}

\subsection{Decomposing absolute homology}\label{3.3}
From now on, we will restrict to the case of regular infinitesimal
character and integral weights. This means that the affine Weyl orbit
in question contains a $\frak g$--dominant integral weight $\la$, so
we can actually start with a finite dimensional complex irreducible
representation $\Bbb V$ of $\frak g$ with lowest weight $-\la$. By
Kostant's theorem (applied to $\frak p\subset\frak g$), the $\frak
p$--dominant weights in the affine Weyl orbit of $\la$ are exactly the
highest weights of the $\frak p$--irreducible summands in $H_*(\frak
p_+,\Bbb V)$. If $\Bbb W$ is one of these summands, then we can
consider $H_*(\frak q_+/\frak p_+,\Bbb W)$. If the lowest weight of
the summand is $-w\cdot\la$ with $w\in W^{\frak p}$, then we can apply
Theorem \ref{thm2.9} to see that this relative homology is the direct
sum of one copy of each of the $\frak q$--irreducible representations
with negative of the lowest weight contained in the affine $W^{\frak
  q}_{\frak p}$--orbit of $w\cdot\la$. In view of Proposition
\ref{prop3.2}, this implies the following theorem in the complex case,
the real case then follows by complexification.

\begin{thm}\label{thm3.3}
  Let $\frak q\subset\frak p$ be two nested parabolic subalgebras in a
  real or complex semisimple Lie algebra $\frak g$ and let $\Bbb V$ be
  a completely reducible representation of $\frak g$. Then, as a module
  over $\frak q_0$ and for each $k=0,\dots,\dim(\frak q_+)$, one has
$$
H_k(\frak q_+,\Bbb V)\cong \oplus_{i+j=k}H_i(\frak q_+/\frak p_+,H_j(\frak
p_+,\Bbb V)).
$$
\end{thm}

As it stands, this is an abstract isomorphism deduced from coincidence
of highest weights of irreducible components, and initially it is
unclear how to obtain explicit maps realizing such an isomorphism. The
rest of this article is devoted to giving a construction of an explicit
$\frak q$--invariant map inducing this isomorphism. This will be a
crucial ingredient for the application to relative BGG sequences.

Observe that the statement of Theorem \ref{thm3.3} actually looks like
the result of a collapsed homology--version of a Hochschild--Serre
spectral sequence (compare with Theorem 12.6 in
\cite{users-guide}). Indeed, our description is based on a $\frak
q$--invariant filtration of the spaces $C_*(\frak q_+,\Bbb V)$ in the
standard complex computing $H_*(\frak q_+,\Bbb V)$. However, this
filtration is not compatible with the Lie algebra homology
differential, which we will denote by $\partial^*_{\frak q}$ in what
follows, so we do not obtain a filtered complex. One could also
involve a filtration on $\Bbb V$ to make the standard complex into a
filtered complex and probably use this to obtain an alternative proof
of Theorem \ref{3.3}. However, the resulting filtration looks much
less useful for the applications we have in mind.

There is a natural way to view the spaces $C_*(\frak q_+,\Bbb V)$ as
$k$--linear alternating maps. Since the Killing form of $\frak g$
induces a $\frak q$--equivariant duality between $\frak q_+$ and
$\frak g/\frak q$, we see that we can view $C_*(\frak q_+,\Bbb V)$ as
$L(\La^*(\frak g/\frak q),\Bbb V)$.  Further, since $\frak
p\subset\frak g$ is a $\frak q$--invariant subspace, so is $\frak
p/\frak q\subset\frak g/\frak q$. Having this at hand, we can define
our filtration.

\begin{definition}\label{def3.3}
  Given nested parabolic subalgebras $\frak q\subset\frak
  p\subset\frak g$, a completely reducible representation $\Bbb V$ of
  $\frak g$, and integers $0<\ell\leq k$, we define subspaces $\Cal
  F^\ell_k\subset C_k(\frak q_+,\Bbb V)$ as follows.

  Viewing $\ph\in C_k(\frak q_+,\Bbb V)$ as a $k$--linear, alternating
  map $(\frak g/\frak q)^k\to\Bbb V$, we say that $\ph\in\Cal
  F^\ell_k$ if and only if it vanishes if at least $k-\ell+1$ of its
  entries are from $\frak p/\frak q\subset\frak g/\frak q$. We further
  put $\Cal F^\ell:=\oplus_{k\geq\ell}\Cal F^\ell_k$. 
\end{definition}

\begin{lemma}\label{lem3.3}
  The subspaces $\Cal F^\ell\subset C_*(\frak
  q_+,\Bbb V)$ define a $\frak q$--invariant decreasing filtration,
  i.e.~$\Cal F^0:=C_*(\frak q_+,\Bbb V)\supset\Cal
  F^1\supset\dots\supset \Cal F^r\supset\Cal F^{r+1}=\{0\}$, where
  $r=\dim(\frak g/\frak p)$. 

Moreover, there is a $\frak q$--equivariant surjection 
$$
\Cal F^\ell \to \oplus_{k\geq\ell}L(\La^{k-\ell}(\frak p/\frak
q),\La^\ell\frak p_+\otimes\Bbb V)  
$$
with kernel $\Cal F^{\ell+1}$. 
\end{lemma}
\begin{proof}
  The fact that each $\Cal F^\ell_k$ is $\frak q$--invariant follows
  immediately from the fact that $\frak p/\frak q\subset\frak g/\frak
  q$ is $\frak q$--invariant. If $\ell+1\leq k$ and $\ph\in\Cal
  F^{\ell+1}_k$, then $\ph$ vanishes upon insertion of $k-\ell$
  entries from $\frak p/\frak q$. But then evidently we have
  $\ph\in\Cal F^\ell_k$, so $\Cal F^\ell\supset\Cal F^{\ell+1}$.

Next, take $\ph\in\Cal F^\ell_k$ and view it as a $k$--linear
alternating map $(\frak g/\frak q)^k\to\Bbb V$. Inserting $k-\ell$
elements from $\frak p/\frak q$ into $\ph$, one obtains an
$\ell$--linear alternating map $(\frak g/\frak q)^\ell\to\Bbb V$. By
assumption the resulting map vanishes upon insertion of a single
element from $\frak p/\frak q$, so it descends to an element of
$L(\La^\ell(\frak g/\frak p),\Bbb V)\cong\La^\ell\frak p_+\otimes\Bbb
V$. Hence we have defined a map
\begin{equation}\label{eq:assoc}
\Cal F^\ell_k\to L(\La^{k-\ell}(\frak p/\frak
q),\La^\ell\frak p_+\otimes\Bbb V),
\end{equation}
which is $\frak q$--equivariant by construction, and whose kernel by
definition coincides with $\Cal F^{\ell+1}_k$. 

If $\ell>r=\dim(\frak g/\frak p)=\dim(\frak p_+)$, then the target
space in \eqref{eq:assoc} is trivial, so we see that $\Cal
F^{\ell}=\Cal F^{\ell+1}$ if $\ell>r$. But since $\Cal F^\ell$
evidently is zero for $\ell>\dim(\frak q_+)$, we see that $\Cal
F^{r+1}=\{0\}$.

On the other hand, given an element of the target space in
\eqref{eq:assoc}, we can first choose an arbitrary extension to a
multilinear alternating map defined on $(\frak g/\frak
q)^{k-\ell}$. The values of this map can be interpreted as
$\ell$--linear maps $(\frak g/\frak p)^\ell\to\Bbb V$. Since $\frak
g/\frak p$ is a quotient of $\frak g/\frak q$, they can be viewed as
$\ell$--linear alternating maps defined on $(\frak g/\frak
q)^\ell$. Taking the arguments together and forming the complete
alternation, we obtain an element $C_k(\frak q_+,\Bbb V)$ and a moment
of thought shows that this is contained in $\Cal F^\ell_k$, so the map
in \eqref{eq:assoc} is surjective.
\end{proof}

Next, we have the Lie algebra homology differential $\partial^*_{\frak
  p}:\La^\ell\frak p_+\otimes\Bbb V\to\La^{\ell-1}\frak p_+\otimes\Bbb
V$, which is $\frak p$--equivariant and hence $\frak
q$--equivariant. In particular, the kernel $\ker(\partial^*_{\frak
  p})$ is a $\frak q$--invariant subspace in $\La^\ell\frak
p_+\otimes\Bbb V$. The maps with values in this subspace form a $\frak
q$--invariant subspace in $L(\La^{k-\ell}(\frak p/\frak q),\La^\ell\frak
p_+\otimes\Bbb V)$, and we denote by $\tcf^\ell_k\subset\Cal F^\ell_k$
the preimage of this subspace under the map from Lemma
\ref{lem3.3}. Hence $\tcf^\ell_k\subset\Cal F^\ell_k$ is a $\frak
q$--invariant subspace, which by Lemma \ref{lem3.3} contains $\Cal
F^{\ell+1}_k$, and we define
$\tcf^\ell=\oplus_{k\geq\ell}\tcf^\ell_k$.

Using the $\frak p$--equivariant quotient map $\ker(\partial^*_{\frak
  p})\to H_\ell(\frak p_+,\Bbb V)$ and the map from Lemma
\ref{lem3.3}, we obtain a $\frak q$--equivariant surjection 
\begin{equation}\label{pidef}
\pi:\tcf^\ell_k\to L(\La^{k-\ell}(\frak p/\frak q),H_\ell(\frak p_+,\Bbb
V))\cong \La^{k-\ell}(\frak q_+/\frak p_+)\otimes H_\ell(\frak
p_+,\Bbb V),
\end{equation}
which by construction vanishes on $\Cal
F^{\ell+1}_k\subset\tcf^\ell_k$.

\subsection{Relating absolute and relative homology}\label{3.4}
As a next step, we can clarify the compatibility of the map $\pi$ from
\eqref{pidef} with the Lie algebra homology differential
$\partial^*_{\frak q}:\La^k\frak q_+\otimes\Bbb V\to\La^{k-1}\frak
q_+\otimes\Bbb V$. 

\begin{prop}\label{prop3.4}
  (1) For the Lie algebra homology differential $\partial^*_{\frak
    q}$, it holds that $\ker(\partial^*_{\frak q})\cap\Cal
  F^\ell\subset\tcf^\ell$, $\partial^*_{\frak q}(\Cal
  F^\ell)\subset\Cal F^{\ell-1}$ and $\partial^*_{\frak
    q}(\tcf^\ell)\subset\tcf^\ell$.

  (2) Denoting by $\partial^*_\rho$ the relative Lie algebra homology
  differential acting on the space $\La^{k-\ell}(\frak q_+/\frak
  p_+)\otimes H_\ell(\frak p_+,\Bbb V)$, $\pi\o\partial^*_{\frak q}$
  coincides with $\partial^*_\rho\o\pi$ up to sign.
\end{prop}
\begin{proof}
To prove the statement, we use the decomposition $\frak q_+=(\frak
q_+\cap\frak p_0)\oplus\frak p_+$ from Section \ref{2.1}, noting that
the second summand is an ideal in $\frak q_+$, while the first summand
is a Lie subalgebra isomorphic to $\frak q_+/\frak p_+$. This gives
rise to a bigrading on $C_*(\frak q_+,\Bbb V)$ as
$$ \La^k\frak q_+\otimes\Bbb V=\oplus_{r+s=k}\La^r(\frak q_+\cap\frak
p_0)\otimes\La^s\frak p_+\otimes\Bbb V=:\oplus_{r+s=k}\La^{(r,s)}\frak
q_+\otimes\Bbb V.
$$ 

Since $\frak p_+$ is the annihilator of $\frak p$ in $\frak q$, we see
that $\Cal F^\ell_k=\oplus_{s\geq\ell}\La^{(k-s,s)}\frak
q_+\otimes\Bbb V$ for $k\geq\ell$, and that the projection from Lemma
\ref{lem3.3} corresponds to extracting the component in
$\La^{(k-\ell,\ell)}\frak q_+\otimes\Bbb V$.

Now consider a decomposable element of $\La^{(r,s)}\frak
q_+\otimes\Bbb V$, which we denote as
\begin{equation}\label{elem}
Z_1\wedge\dots\wedge Z_r\wedge W_1\wedge\dots\wedge W_s\otimes v, 
\end{equation}
so the $Z$'s are in $\frak q_+\cap\frak p_0$ and the $W$'s are in
$\frak p_+$. Moreover, the bracket of two $W$'s and the bracket of a
$Z$ with a $W$ lies in $\frak p_+$, while the bracket of two $Z$'s lies
in $\frak q_+\cap\frak p_0$.

Now the formula for the Lie algebra homology differential (compare
with \eqref{eq:codiff}) immediately implies that $\partial^*_{\frak
  q}$ maps this element to the sum of the components with bidegrees
$(r-1,s)$ and $(r,s-1)$. On the one hand, this implies that
$\partial^*_{\frak q}(\Cal F^\ell_k)\subset\Cal F^{\ell-1}_{k-1}$.

On the other hand, we see that this construction makes $C_*(\frak
q_+,\Bbb V)$ into a double complex, and we denote by $\partial^*_1$
and $\partial^*_2$ the two components of $\partial^*_{\frak q}$. Then
  we obtain the usual relations $(\partial^*_i)^2=0$ for $i=1,2$ and
  $\partial^*_1\partial^*_2=-\partial^*_2\partial^*_1$. From the
  definition it also follows that the component of degree $(r,s-1)$
  consists of the summands in which a $W$ acts on $v$ and the summands
  containing the brackets of two $W$'s. This easily implies that
  $\partial^*_2=(-1)^r\id\otimes\partial^*_{\frak p}$. 

From the construction in Section \ref{3.3} it is also clear that
$$
\tcf^\ell_k \subset\Cal F^\ell_k\cong\oplus _{\ell\leq s\leq
  k}\La^{(k-s,s)}\frak q_+\otimes\Bbb V
$$ 
exactly consists of those elements, for which the component in
$\La^{(k-\ell,\ell)}\frak q_+\otimes\Bbb V$ lies in the kernel of
$\id\otimes\partial^*_{\frak p}$. This readily implies the first
statement in (1) as well as $\partial^*_{\frak
  q}(\tcf^{\ell}_k)\subset\Cal F^\ell_{k-1}$. Moreover, an element
$\ph\in\tcf^\ell_k$ can be written as $\ph=\ph_1+\ph_2$ with
$\partial^*_2(\ph_1)=0$ and $\ph_2\in\Cal F^{\ell+1}_k$. But then
$\partial^*_{\frak q}(\ph)$ is congruent to
$\partial^*_1(\ph_1)+\partial_2^*(\ph_2)$ modulo $\Cal
F^{\ell+1}_{k-1}$. Since both these summands lie in
$\ker(\partial^*_2)$ we conclude that $\partial^*_{\frak
  q}(\ph)\in\tcf^\ell_{k-1}$, which completes the proof of (1).

\medskip

Returning to the decomposable element of $\La^{(r,s)}\frak
q_+\otimes\Bbb V$ from \eqref{elem}, we next observe that there is a
natural action of the Lie algebra $\frak q_+\cap\frak p_0$ on
$\La^s\frak p_+\otimes\Bbb V$, which we denote by
$\bullet$. Explicitly, $Z\bullet(W_1\wedge\dots\wedge W_j\otimes
v)$ is given by 
$$ 
W_1\wedge\dots\wedge W_j\otimes Z\cdot v
+\textstyle\sum_iW_1\wedge\dots\wedge[Z,W_i]\wedge\dots\wedge W_j\otimes v.
$$
This easily implies that, in terms of this action, $\partial^*_1$ maps
the element \eqref{elem} to
\begin{align*}
\textstyle\sum_{i=1}^r&(-1)^iZ_1\wedge\dots\widehat{Z_i}\dots\wedge Z_r\otimes
Z_i\bullet (W_1\wedge\dots\wedge W_j\otimes v)\\
+&\textstyle\sum_{i<j}(-1)^{i+j}[Z_i,Z_j]\wedge
Z_1\wedge\cdots\widehat{Z_i}\cdots\widehat{Z_j}\cdots\wedge Z_r\wedge
W_1\wedge\dots\wedge W_s\otimes v.
\end{align*}
Clearly, $\ker(\partial^*_{\frak p})\subset \La^s\frak p_+\otimes\Bbb
V$ is invariant under the action $\bullet$, and projecting further to
$H_s(\frak p_+,\Bbb V)$ the action $\bullet$ corresponds to the
natural action of $\frak q_+/\frak p_+$. The definition of
$\partial^*_\rho$ then implies the statement in (2).
\end{proof}

Using this, we can now construct an explicit map relating absolute and
relative homology groups. Consider the intersection
$\ker(\partial^*_{\frak q})\cap\Cal F^\ell_k$, which is a $\frak
q$--submodule in $C_k(\frak q_+,\Bbb V)$. By part (1) of Proposition
\ref{prop3.4}, this is contained in $\tcf^\ell_k$ so $\pi$ is defined
on this subspace. By part (2) of Proposition \ref{prop3.4}, this
restriction has values in $\ker(\partial^*_\rho)\subset
\La^{k-\ell}(\frak q_+/\frak p_+)\otimes H_\ell(\frak p_+,\Bbb
V)$. Hence we can postcompose with the canonical map to relative
homology to obtain a $\frak q$--equivariant map
\begin{equation}\label{Pidef}
\Pi:\ker(\partial^*_{\frak q})\cap\Cal F^\ell_k\to 
H_{k-\ell}(\frak q_+/\frak p_+,H_\ell(\frak p_+,\Bbb V)).
\end{equation}

Using this, we can formulate our final result. 
 
\begin{thm}\label{thm3.4}
For $\ell\leq k$, the map $\Pi$ vanishes on $\im(\partial^*_{\frak q})\cap\Cal
  F^\ell_k$ and descends to a $\frak q$--equivariant surjection
$$ 
H_k(\frak q_+,\Bbb V)\to H_{k-\ell}(\frak q_+/\frak
  p_+,H_\ell(\frak p_+,\Bbb V)).
$$
\end{thm}
\begin{proof}
Via complexifications and direct sums, it suffices to prove this in
the case that the Lie algebras are complex and that $\Bbb V$ is a
complex irreducible representation of $\frak g$. Fix an irreducible
component $\Bbb W$ of $H_\ell(\frak p_+,\Bbb V)$. If $-\la$ is the
lowest weight of $\Bbb V$, then the lowest weight of $\Bbb W$ must be
$-w_2\cdot\la$ for an element $w_2\in W^{\frak p}$ of length
$\ell$. By Theorem \ref{thm2.9}, the $\frak q$--representation
$H_*(\frak q_+/\frak p_+,\Bbb W)$ is completely reducible with each
irreducible component occurring with multiplicity one. Moreover, the
lowest weights of these components are exactly the weights $-w_1\cdot
w_2\cdot \la$ for $w_1\in W^{\frak q}_{\frak p}$.

From Proposition \ref{prop3.2}, we know that $w:=w_1w_2$ lies in
$W^{\frak q}$ and denoting by $k$ its length, we see that $w_1$ has
length $k-\ell$. By Kostant's theorem, $H_k(\frak q_+,\Bbb V)$ contains
an irreducible representation of $\frak q$ with lowest weight
$-w\cdot\la$, and we can consider the corresponding lowest weight
vector $\ph\in\La^k\frak q_+\otimes\Bbb V$. This is the wedge product
of root vectors $e_\al\in\frak g_{\al}$ for each $\al\in\Ph_w$
tensorized with a certain weight vector of $\Bbb V$. Further,
$\Ph_w\subset\Delta^+(\frak q_+)$ and
$\Ph_{w_1}=\Ph_w\cap\Delta^+(\frak p_0)$, and this has $k-\ell$
elements. Hence in the language of the proof of Proposition
\ref{prop3.4}, the lowest weight vector is contained in
$\La^{(k-\ell,\ell)}\frak q_+\otimes\Bbb V$, so it lies in $\Cal
F^\ell_k$ but not in $\Cal F^{\ell+1}_k$. Of course, it also lies in
$\ker(\partial^*_{\frak q})$.

Hence $\ph$ has non--trivial image in the quotient $\Cal
F^\ell_k/\Cal F^{\ell+1}_k$ and its image under the map from
\eqref{eq:assoc} is a decomposable element of $\La^{k-\ell}(\frak
q_+/\frak p_+)\otimes\ker(\partial^*_\rho)$. The second component of
this is a weight vector and by the multiplicity--one result in
Kostant's theorem (applied to $H_*(\frak p_+,\Bbb V)$), our element
has to have nontrivial image in $\La^{k-\ell}(\frak q_+/\frak
p_+)\otimes H_\ell(\frak p_+,\Bbb V)$. More precisely, this image must
be the tensor product of a decomposable element of $\La^{k-\ell}(\frak
q_+/\frak p_+)$ with the lowest weight vector for the irreducible
component $\Bbb W$ from above. By part (2) of Proposition
\ref{prop3.4} our element lies in $\ker(\partial^*_\rho)$ and by the
multiplicity--one part of Theorem \ref{thm2.9}, it must be a harmonic
element contained in the component in the homology corresponding to
$w_1$. This shows that $\Pi(\ph)$ is a lowest weight vector for the
irreducible component in question, so $\Pi$ is surjective.

To complete the proof, it remains to show that $\Pi$ vanishes on
$\im(\partial^*_{\frak q})$. But if $\Pi$ would be non--zero on that
subspace, then the image of $\Pi|_{\im(\partial^*_{\frak q})}$ would
contain one of the $\frak q_0$--irreducible components of $H_*(\frak
q_+/\frak p_+,H_\ell(\frak p_+,\Bbb V))$. But the lowest weights of
each of these components is a lowest weight of an irreducible
component of $H_*(\frak q_+,\Bbb V)$, so a weight vector of that
weight cannot be contained in $\im(\partial^*_{\frak q})$ by the
multiplicity--one part of Kostant's theorem.
\end{proof}


\end{document}